\documentclass[runningheads,a4paper]{llncs}
\usepackage[utf8]{inputenc}
\usepackage[T1]{fontenc}
\usepackage {amssymb}
\usepackage {amsmath}
\usepackage {amsfonts}
\usepackage{algorithm}
\usepackage {latexsym}
\usepackage {paralist}
\usepackage{listings}
\usepackage{rotating}
\usepackage{url}
 \newcommand\ForAuthors[1]
 {\par\smallskip                     
  \begin{center}
   \fbox
   {\parbox{0.9\linewidth}
    {\raggedright\sc--- #1}
   }
  \end{center}
  \par\smallskip                     %
 }
\usepackage{color}

\newcommand{\comment}[1]{}
\newcommand{\rr}{\mathbb R}
\newcommand{\rd}{\rr^d}
\newcommand{\rdm}{\rr^{d+m}}
\newcommand{\Id}{\operatorname{Id}}
\newcommand{\rea}{\mathcal{R}}
\newcommand{\inset}{\mathcal{S}}
\newcommand{\bdu}{\mathcal{U}}
\newcommand{\state}{\mathcal{X}}
\def\norm#1{\mbox{$\| #1 \|$}}

\def\nn{\mathbb N}
\def\xu{\begin{pmatrix} x\\ u\end{pmatrix}} 
\def\uxu{\begin{pmatrix} 1\\ x \\ u\end{pmatrix}}

\title{Automatic Synthesis of Piecewise Linear Quadratic Invariants for Programs%
\thanks{This work has been partially supported by an RTRA/STAE BRIEFCASE project
  grant, the ANR projects INS-2012-007 CAFEIN, and ASTRID VORACE.}}

\titlerunning{Piecewise Quadratic Invariant}  
\author{Assalé Adjé
 \and Pierre-Loïc Garoche
 }
\authorrunning{Adjé, Garoche} 

\institute{Onera, the French Aerospace Lab, France\\Université de Toulouse,
  Toulouse, France\\ \email{firstname.lastname@onera.fr}}

\begin{document}
\mainmatter

\maketitle

\begin{abstract}
Among the various critical systems that worth to be formally analyzed, a wide
set consists of controllers for dynamical systems. Those programs typically
execute an infinite loop in which simple computations update internal states
and produce commands to update the system state.
Those systems are yet hardly analyzable by available static analysis
method, since, even if performing mainly linear computations, the computation of a safe
set of reachable states often requires quadratic invariants.

In this paper we consider the general setting of a piecewise affine program;
that is a program performing different affine updates on the system depending on
some conditions. This typically encompasses linear controllers with
saturations or controllers with different behaviors and performances activated on
some safety conditions.

Our analysis is inspired by works performed a decade ago by Johansson et al, and
Morari et al, in the control community. We adapted their method focused
on the analysis of stability in continuous-time or discrete-time settings to fit the static analysis 
paradigm and the computation of invariants, that is over-approximation of reachable sets
using piecewise quadratic Lyapunov functions.

\keywords{formal verification, static analysis, piecewise affine systems,
  piecewise quadratic lyapunov functions}

\end{abstract}


\section{Introduction}

With the success of Astrée~\cite{DBLP:journals/sigsoft/BertraneCCFMMR11}, static
analysis in general and abstract interpretation in particular are now seriously
considered by industrials from the critical embedded system community, and more
specifically by the engineers developping and validation controllers. The
certification norms concerning the V\&V of those software have also evolved and
now enable the use of such methods in the developement process.

These controller software are meant to perform an infinite loop in which values
of sensors are read, a function of inputs and internal states is computed, and
the value of the result is sent to actuators.  In general, in the most critical
applications, the controllers used are based on a simple linear update with
minor non linearities such as saturations, i.e. enforcing bounds, or specific
behaviors when some conditions are met. The controlled systems range from
aircraft flight commands, guidance algorithms, engine control from any kind of
device optimizing performance or fuel efficiency, control of railway
infrastructure, fan control in tunnels, etc.

It is therefore of outmost importance to provide suitable analyses to verify
these controllers. One of the approach is to rely on quadratic invariants, such
as the digital filters abstract domain of Feret~\cite{DBLP:conf/esop/Feret04}, since, according
to Lyapunov theorem, any globally asymptotically stable linear system admits a
quadratic Lyapunov function. This theorem does not hold in presence of
disjunction, such as saturations.

In static analysis, dealing with disjunction is an import concern. When the join
of two abstract element is imprecise, one can consider the disjunctive
completion of the domain~\cite{DBLP:conf/slp/FileR94}. This process enriches the
set of abstract elements with new ones, but the cost, i.e. the number of new
elements, could be exponential in the number of initial elements.  Concerning
relation abstract domains, one should mention the tropical polyhedra of
Allamigeon~\cite{AllamigeonThesis} in which an abstract element characterizes a finite disjunction of
zones~\cite{DBLP:conf/pado/Mine01}. However concerning quadratic properties, no static
analysis actually performs the automatic computation of disjunctive quadratic
invariants.

The goal of this paper is to propose such a computation: produce a disjunctive
quadratic invariant as a sub-level of a piecewise quadratic Lyapunov function.

\paragraph{Related works}

Most relational abstractions used in the static analysis community rely on a
linear representation of relationship between variables,
e.g. polyhedra~\cite{DBLP:conf/popl/CousotH78},
octagons~\cite{DBLP:journals/lisp/Mine06},
zonotopes~\cite{DBLP:conf/cav/GhorbalGP09} are not join-complete. As
mentionned above, the tropical polyhedra domain~\cite{AllamigeonThesis} admits some
disjunctions since it characterizes a family of properties encoded as finite
disjunction of zones.

Concerning non linear properties, the need for quadratic invariant was addressed
a decade ago with ellipsoidal abstract domains for simple linear
filters~\cite{DBLP:conf/esop/Feret04} and more recently for non linear template
domains~\cite{DBLP:conf/esop/ColonS11} and policy iteration based static
analysis~\cite{DBLP:journals/jsc/GawlitzaSAGG12}.

More recently, techniques used in the control community have been used to
synthesize appropriate quadratic templates using SDP solvers and Lyapunov
functions~\cite{DBLP:conf/hybrid/RouxJGF12}.

The proposed technique addresses a family of systems well beyond the ones handled
by the mentionned methods. In general, a global quadratic invariant is not
enough to bound the reachable value of the considered systems, hence none of
these could succeed.
  
On the control community side, Lyapunov based analysis are typically used to
show the good behavior of a controlled system: it is globally asymptotically
stable (GAS), i.e. when time goes to infinity the trajectories of the system goes to
$0$. Since about a decade SDP solvers, i.e. convex optimzation algorithms for
semi-definite programming, have reached a level of maturity that enable their use to
compute quadratic Lyapunov functions. On the theory side, variants of quadratic Lyapunov 
functions such as the papers motivating our work --
Johansson and Rantzer~\cite{847100,1205102} as well as Mignone, Ferrari-Trecate and
Morari~\cite{912814} -- addressed the study of piecewise linear systems for proving
the GAS property.

In general, computing a safe superset of reachable states as needed when
performing static analysis, is not a common question for control theorist. They
would rather address the related notions of controllabilty or stability under
pertubations. In most case, either the property considered, or the technique used, 
relies on the existence of a such a bound over reachable state; which we aim to compute
in static analysis.

\paragraph{Contributions}

Our contribution is threefold and based on the method of Johansson and Mignone used to prove the GAS property
of a piecewise linear system:
\begin{itemize}
\item we detailed the method in the discrete setting, computing a piecewise
  quadratic Lyapunov function of a \textit{discrete-time system};
\item we adapted it to compute an invariant over reachables states of the analyzed system;
\item we showed the applicability of the proposed method to a wide set of generated examples.
\end{itemize}

\paragraph{Organisation of the paper}

The paper is structured as follow. Section~\ref{sec:statement} introduces the
kind of programs considered. Section~\ref{sec:contrib} details our version of
the piecewise quadratic Lyapunov function as well as the characterization of
invariant sets. Section~\ref{sec:exp} presents the experimentations while
Section~\ref{sec:conclu} concludes and opens future direction of research.


\section{Problem statement} 
\label{sec:statement}

In this paper, we are interested in analyzing computer science programs. The programs we consider 
here are composed of a single loop with possibly a complicated switch-case type loop body. Our switch-case 
loop body is supposed to be written as a nested sequence of \emph{ite}
statements, or a \emph{switch $c1\to inst1; c2\to instr2; c3\to instr3$}. 
Moreover, we suppose that the analyzed programs are written in affine arithmetics. Consequently, the 
programs analyzed here can be interpreted as piecewise affine discrete-time systems.     
Finally, we reduce the problem to compute automatically an overapproximation of the reachable states of 
a piecewise affine discrete-time system. The term piecewise affine means that there exists 
a polyhedral partition $\{X^i,i\in I\}$ of the state-input space $\state\subseteq \rdm$ such that for all $i\in I$, 
the dynamic of the system is affine and represented by the following relation :
\begin{equation}
\label{pwa}
x_{k+1}=A^i x_k+B^i u_k+b^i, i\in I, k\in\nn
\end{equation}
where $A^i$ is a $d\times d$ matrix, $B^i$ a $d\times m$ matrix and $b^i$ a vector of $\rd$.
The variable $x\in\rd$ refers to the state variable and $u\in\rr^m$ refers to some input variable.

For us, a polyhedral partition is a family of convex polyhedra which partitions the state-input space 
i.e. $\state=\bigcup_{i\in I}X^i\subseteq \rdm$ such that $X^i\cap X^j=\emptyset$ for all $i,j\in I$, $i\neq j$. 
From now on, we call $X^i$ \emph{cells}. Cells $\{X^i\}_{i\in I}$ are convex polyhedra which can contain 
both strict and weak inequalities. Cells can be represented by a $n_i\times (d+m)$ matrix $T^i$ 
and $c^i$ a vector of $\rr^{n_i}$. We denote by $\mathbb{I}_i^s$ the set of indices which represent strict 
inequalities for the cell $X^i$, denote by $T_s^i$ and $c_s^i$ the parts of $T^i$ and $c^i$ 
corresponding to strict inequalities and by $T_w^i$ and $c_w^i$ the one corresponding to weak inequalities. 
Finally, we have the matrix representation given by Formula~\eqref{polyhedra}.  
\begin{equation}
\label{polyhedra}
X^i=\left\{\xu\in\rdm \left| T_s^i\xu \ll c_s^i,\ T_w^i\xu \leq c_w^i\right\}\right.
\end{equation}
We insist on the notation: $y\ll z$ means that for all coordinates $l$, $y_l<z_l$ and $y\leq z$ means that for all 
coordinates $l$, $y_l\leq z_l$.     

We will need homogeneous versions of laws and thus introduct the $(1+d+m)\times (1+d+m)$ 
matrices $F^i$ defined as follows:
\begin{equation}
\label{homogeneouslaw}
F^i=\begin{pmatrix} 1 & 0_{1\times d} & 0_{1\times m} \\
                    b^i & A^i & B^i\\ 
                    0 & 0_{m\times d} & \Id_{m\times m}
    \end{pmatrix}
\end{equation}
The system defined in Equation~\eqref{pwa} can be rewritten as 
$(1,x_{k+1},u_{k+1})^\intercal =F^i(1,x_{k+1},u_k)$.
Note that we introduce a "virtual" dynamic law $u_{k+1}=u_k$ on the input variable in Equation~\eqref{homogeneouslaw}. 
In the point of view of set invariance computation, we will see that it remains to consider 
such dynamic law. Indeed we suppose that the input is bounded and we write 
$u_k\in \bdu$ for all $k\in\nn$ with $\bdu$ is a nonempty compact set (polytope). 


We are interested in proving that the reachable states $\rea$ is bounded and a proof of this statement can be 
expressed by directly computing $\rea$ that is:
\[
\rea=\{ y\in\rd\mid \exists\ k\in \nn, \exists\ i\in I,\ \exists\, u_k\in \bdu,\ y=A^i x_k+B^i u_k+b^i\}\cup \{x_0\}
\]
and prove that this set is bounded. We can also compute an overapproximation of $\rea$ from a set 
$\inset\subseteq \rdm$ such that $(x_0,u_0)\in \inset$, $\rea\times\bdu\subseteq \inset$ and $\inset$ is an inductive 
invariant in the sense of:
\[
(x,u)\in \inset \implies (A^i x+B^i u+b^i,u)\in \inset,\ \forall\, i\in I\enspace .
\]
Indeed, by induction since $(x_0,u_0)$ belongs to $\inset$, $(x_k,u_k)\in\inset$ for all $k\in\nn$. Since 
every image of the dynamic of the system stays in $\inset$, a reachable state $(y,u)$ belongs to 
$\inset$. Finally, if we prove that $\inset$ is bounded then $\rea$ is also bounded.

Working directly on sets can be difficult and usually invariant sets are computed as a sublevel 
of some function to find. For (convergent) discrete-time linear systems, it is classical to compute 
ellipsoidal overapproximation of reachable states. Indeed, sublevel sets of Lyapunov functions are 
invariant set for the analyzed linear system and to compute an ellipsoid containing the initial 
states provides an overapproximation of reachable states. 
Initially, Lyapunov functions are used to prove quadratic asymptotic stability. In this paper, we use 
an analogue of Lyapunov functions for piecewise affine systems to compute directly an overapproximation 
of reachable states.

  \allowdisplaybreaks
\begin{example}[Running example]
Let us consider the following program. It is constituted by a single while loop with 
two nested conditional branches in the loop body. 

\begin{lstlisting}[mathescape=true,frame=single]
(x,y)$\in [-9,9]\times [-9,9]$;
while(true)
   ox=x;
   oy=y;
   read(u); \\$u\in[-3,3]$
   if (-9*ox+7*y+6*u<5){
      if(-4*ox+8*oy-8*u<4){
         x=0.4217*ox+0.1077*oy+0.5661*u;
         y=0.1162*ox+0.2785*oy+0.2235*u-1;
         }
      else { \\4*ox-8*oy+8*u<-4
         x=0.4763*ox+0.0145*oy+0.9033*u;
         y=0.1315*ox+0.3291*oy+0.1459*u+9;
         }
      } 
   else { \\9*ox-7*y-6*u<-5
      if(-4*ox+8*oy-8*u<4){
         x=0.2618*ox+0.1107*oy+0.0868*u-4;
         y=0.4014*ox+0.4161*oy+0.6320*u+4;
         }
      else { \\4*ox-8*oy+8*u<-4
         x=0.3874*ox+0.00771*oy+0.5153*u+10;
         y=0.2430*ox+0.4028*oy+0.4790*u+7;      
         }
      }
\end{lstlisting}

The initial condition of the piecewise affine systems is $(x,y)\in[-9,9]\times [-9,9]$
and the polytope where the input variable $u$ lives is $\mathcal{U}=[-3,3]$.

We can rewrite this program as a piecewise affine discrete-time dynamical systems using
our notations. We give details on the matrices $T_s^i$ and $T_w^i$ and vectors $c_s^i$ and $c_w^i$ 
(see Equation~\eqref{polyhedra}) which characterize the cells and on the matrices $F^i$  
representing the homogeneous version (see Equation~\eqref{homogeneouslaw}) of affine laws 
in the cell $X^i$. 
{\footnotesize
\begin{align*}
F^1= \begin{pmatrix}
  1 & &0  & & 0 & & 0\\
  0  & & 0.4217 & &  0.1077 & & 0.5661 \\
  -1 & & 0.1162 & &  0.2785 & & 0.2235 \\
  0 & & 0 & &  0 & &  1
\end{pmatrix}, &
\left\{
\begin{array}{l}
T_s^1=\begin{pmatrix} -9 & & 7 & & 6\\ -4 & & 8 & & -8\end{pmatrix}\\
\\
c_s^1=(5\ 4)^\intercal
\end{array}
\right.
,& 
\left\{
\begin{array}{l}
T_w^1=\begin{pmatrix}0 & & 0 & & 1\\ 0 & & 0 & & -1\end{pmatrix}\\ 
\\
c_w^1=(3\ 3)^\intercal
\end{array}
\right.
\\
\\
F^2= \begin{pmatrix}
1 & &0  & & 0 & & 0\\
0 & & 0.4763 & & 0.0145 & &  0.9033 \\
9 & &  0.1315 & & 0.3291 & &  0.1459 \\
0 & & 0 & & 0 & & 1 \\
\end{pmatrix}, &
\left\{
\begin{array}{l}
T_s^2=\begin{pmatrix} -9 & & 7 & & 6\end{pmatrix}\\
\\
c_s^2=5 
\end{array}\right. ,
& 
\left\{
\begin{array}{l}
T_w^2=\begin{pmatrix} 4 & & -8 & & 8\\
0 & & 0 & & 1\\ 0 & & 0 & & -1\end{pmatrix}\\
\\
c_w^2=(-4\ 3\ 3)^\intercal\end{array}\right.\\
\\
F^3= \begin{pmatrix}
  1 & &0  & & 0 & & 0\\
  -4  & & 0.2618 & &  0.1177 & & 0.0868 \\
  4 & & 0.4014 & &  0.4161 & & 0.6320 \\
  0 & & 0 & &  0 & &  1
\end{pmatrix}, & 
\left\{
\begin{array}{l}
T_s^3=\begin{pmatrix} -4 & & 8 & & -8\end{pmatrix}\\
\\
c_s^3=4\end{array}\right. ,
& 
\left\{
\begin{array}{l}
T_w^3=\begin{pmatrix}9 && -7 && -6 \\0 & & 0 & & 1\\ 0 & & 0 & & -1\end{pmatrix}\\
\\
c_w^2=(-5\ 3\ 3)^\intercal
\end{array}\right.\\
\\
F^4= \begin{pmatrix}
  1 & &0  & & 0 & & 0\\
  10  & & 0.3874 & &  0.0771 & & 0.5153 \\
  7 & & 0.2430 & &  0.4028 & & 0.4790 \\
  0 & & 0 & &  0 & &  1
\end{pmatrix}, & 
\left\{
\begin{array}{l}
T_w^4=\begin{pmatrix}9 & & -7 & & -6\\ 4 & & -8 & & 8 \\
0 & & 0 & & 1\\ 0 & & 0 & & -1\end{pmatrix}\\
\\
c_w^4=(-5\ -4\ 3\ 3)^\intercal
\end{array}\right.
\end{align*}
}
\end{example}



\section{Invariant computation}
\label{sec:contrib}
In~\cite{1205102,912814}, the authors propose a method to prove stability of piecewise affine 
dynamical discrete-time systems. The method is a generalisation of Lyapunov stability equations 
in the case where affine laws defining the system depend on the current state. Let $A$ be a $d\times d$
matrix and let $x_{k+1}=Ax_k,\ k\in\nn,\ x_0\in\rd$ be a linear dynamical system. We recall that $L$ is a quadratic 
Lyapunov function iff there exists a $d\times d$ symmetric matrix $P$ such that $L(x)=x^\intercal P x$ for all 
$x\in\rd$ and $P\succ 0\text{ and } P-A^\intercal P A\succ 0$.
The notation $P \succ 0$ means that $P$ is positive definite i.e. $x^\intercal P x>0$ for all $x\in\rd,\ x\neq 0$ and 0 for $x=0$. 
We will denote by $Q \succeq 0$ when $Q$ is positive semidefinite i.e. $x^\intercal P x\geq 0$ for all $x\in\rd$. 
Positive definite matrices characterize square of norm on $\rd$. A Lyapunov function allows to prove the stability 
by the latter fact : the norm (associated to the Lyapunov function) of the states $x_k$ decreases along the time.
In switched system, similarly to the classical case, we exhibited a positive matrix (square norm) to prove that
the trajectories decrease along the time. The main difficulty in the switched case is the fact that we change the laws 
and we must decrease whenever a transition from one cell to other is fired. Moreover, we only require the norm to be 
local i.e. positive only where the law is used. 
\subsection{Quadratization of cells}
\label{quadratizationsub}
We recall that for us local means that true on a cell and thus true on a polyhedron. Using the homogenous 
version of a cell, we can define local positiveness on a polyhedral cone. Let $Q$ be a $d\times d$ 
symmetric matrix and $M$ be a $n\times d$ matrix. Local positivity in our case means that 
$My\geq 0\implies y^\intercal Q y\geq 0$. The problem will be to write the local positivity as a constraint 
without implication. The problem is not new (e.g. the survey paper~\cite{SurveyKpositivity}). The paper~\cite{Martin1981227} 
proves that local positivity is equivalent, when $M$ has a full row rank, to $Q-M^\intercal C M\succeq 0$ where $C$ is 
a copositive matrix i.e. $x^\intercal C x\geq 0$ if $x\geq 0$. First in general (when the rank of $M$ is not necessarily  
equal to its number of rows), note that if $Q-M^\intercal C M\succeq 0$ for some copositive matrix $C$ then $Q$ satisfies 
$My\geq 0\implies y^\intercal Q y\geq 0$. Secondly every matrix $C$ with nonnegative entries is copositive. Since 
copositivity seems to be as difficult as local positivity to handle, we will restrict copositive matrices to 
be matrices which nonnegative entries.  
The idea is instead of using cells as polyhedral cones, we use a quadratization of cells by 
introducing nonnegative entries and we will define the quadratization of a cell $X^i$ by:
\begin{equation}
\label{quadratization}
\overline{X^i}=\left\{ \xu\in\rdm\left| \uxu^\intercal {E^i}^\intercal W^i E^i\uxu\geq 0\right\}\right.
\end{equation}
where $W^i$ is a $(1+n_i)\times (1+n_i)$ symmetric matrix with nonnegative entries and $E^i=\begin{pmatrix} E_s^i\\ E_w^i\end{pmatrix}$ 
with $E_s^i=\begin{pmatrix} 1 & & 0_{1\times (d+m)}\\ c_s^i & & -T_s^i\end{pmatrix}$ and 
$E_w^i=\begin{pmatrix} c_w^i & &-T_w^i\end{pmatrix}$. Recall that $n_i$ is the number 
of rows of $T^i$. The matrix $E^i$ is thus of the size $n_i+1\times (1+d+m)$. 
The goal of adding the row $(1,0_{1\times (d+m)})$ is to avoid to add the opposite of a vector of $X^i$ 
in $\overline{X^i}$. Indeed without this latter vector $\overline{X^i}$ would be symmetric. We illustrate 
this fact at Example~\ref{reason_of_one}. Note that during optimisation process, matrices $W^i$ will be decision variables. 
\begin{example}[The reason of adding the row $(1,0_{1\times (d+m)})$]
\label{reason_of_one}
Let us take the polyedra $X=\{x\in\rr\mid x\leq 1\}$. Using our notations, 
we have $X=\{x\mid M(1\ x)^\intercal \geq 0\}$ with $M=(1\ -1)$. 
Let us consider two cases, the first one without adding the row 
and the second one using it.

Without any modification, the quadratization of $X$ relative to a nonnegative real $W$ is 
$X'=\{x\mid (1\ x)M^\intercal W M (1\ x)^\intercal \geq 0\}$.
But $(1\ x)M^\intercal W M (1\ x)^\intercal=W (1\ x) (1\ -1)^\intercal (1\ -1) (1\ x)^\intercal=2W(1-x)^2$. 
Hence $X'=\rr$ for all nonnegative real $W$.

Now let us take $E=\begin{pmatrix} 1 & 0 \\ 1 & -1\end{pmatrix}$. The quadratization as defined by 
Equation~\eqref{quadratization} relative to a $2\times 2$ symmetric matrix $W$ with nonnegative coefficients 
is $\overline{X}=\{x\mid (1\ x)E^\intercal W E (1\ x)^\intercal \geq 0\}$.
We have: 
\[
(1\ x)\begin{pmatrix} 1 & 1 \\ 0 & -1\end{pmatrix}\begin{pmatrix} w_1 & w_3 \\ w_3 & w_2\end{pmatrix} 
\begin{pmatrix} 1 & 0 \\ 1 & -1\end{pmatrix}(1\ x)^\intercal=w_1+2w_3(1-x)+w_2(1-x)^2\enspace .
\] 
To take a matrix $W$ such that $w_2=w_1=0$ and $w_3>0$ implies that $\overline{X}=X$.
\end{example}
Now we introduce an example of the quadratization of the cell $X^1$ for our running example.
\begin{example}
Let us consider the running example and the cell $X^1$. We recall that 
$X^1$ is characterizd by the matrices and vectors:
\[
\begin{array}{ccc}
\left\{
\begin{array}{l}
T_s^1=\begin{pmatrix} -9 & & 7 & & 6\\ -4 & & 8 & & -8\end{pmatrix}\\
\\
c_s^1=(5\ 4)^\intercal
\end{array}
\right.
, & 
\left\{
\begin{array}{l}
T_w^1=\begin{pmatrix}0 & & 0 & & 1\\ 0 & & 0 & & -1\end{pmatrix}\\ 
\\
c_w^1=(3\ 3)^\intercal
\end{array}
\right. 
\text{ and }& 
E^1=\begin{pmatrix}1 & & 0 & & 0 & &0 \\
                              5 & & 9 & & -7 & & -6\\
                              4 & &4 & & -8 && 8\\
                              3 & & 0 & & 0 & & -1\\ 
                                3 & & 0 & & 0 & & 1
               \end{pmatrix}
\end{array}
\]
As suggested we have added the row $(1,0_{1\times 3})$.  
Take for example the matrix:
\[
W^1=\begin{pmatrix}
63.0218 & 0.0163  &  0.0217  & 12.1557  &  8.8835\\
0.0163  &  0.0000  &  0.0000  &  0.0267   & 0.0031\\
0.0217  &  0.0000   & 0.0000  &  0.0094   & 0.0061\\
12.1557  &  0.0267  &  0.0094 &   4.2011  & 59.5733\\
8.8835   & 0.0031   & 0.0061  & 59.5733  &  3.0416
\end{pmatrix}
\] 
We have $\overline{X^1}=\{(x,y,u)\mid (1,x,y,u)E^1 W^1 E^1(1,x,y,u)^\intercal\}\supseteq X^1$. 
In Section~\ref{sec:exp}, we will come back on the generation of $W^1$.
\end{example}
Local positivity of quadratic forms will also be used when a transition from a cell to an other is fired . 
For the moment, we are interested in the set of $(x,u)$ such that $(x,u)\in X^i$ and whose 
the image is in $X^j$ and we denote by $X^{ij}$ the set:
\[
\left\{\xu\in\rdm\left| \xu\in X^i \text{ and } (A^ix+B^iu+b^i,u)\in X^j\right\}\right.
\] 
for all pairs $i,j\in I$. Note that in~\cite{912814}, the authors take
into account all pairs $(i,j)$ such that there exists a state $x_k$ at moment $k$ in $X^i$ 
and the image of $x_k$ that is $x_{k+1}$ is in $X^j$. We will discuss in Subsection~\ref{switchcontribution}
the computation or a reduction to possible switches using linear programming as suggested in~\cite{BisEtal:2005:IFA_2030}.
To construct a quadratization of $X^{ij}$, we use the same approach than before by introducing a $(1+n_i+n_j)\times (1+n_i+n_j)$ 
symmetric matrix $U^{ij}$ with nonnegative entries to get a set $\overline{X^{ij}}$ defined as:
\begin{equation}
\label{switchquad}
\overline{X^{ij}}=\left\{ \xu\in\rdm\left| \uxu^\intercal {E^{ij}}^\intercal U^{ij} E^{ij}\uxu\geq 0\right\}\right.
\end{equation}
where $E^{ij}=\begin{pmatrix} E_s^{ij} \\ E_w^{ij} \end{pmatrix}$ with

\begin{equation}
\label{switchmat}
\begin{array}{c}
E_s^{ij}=
\begin{pmatrix}
1 & & 0_{1\times (d+m)}\\
c_s^i & &-T_s^i\\
c_s^j-T_s^j\begin{pmatrix} b^i\\0\end{pmatrix} & &-T_s^j\begin{pmatrix} A^i & B^i\\ 0_{d\times m} & \Id_{m\times m}\end{pmatrix}
\end{pmatrix}\\
\text{and}\\
E_w^{ij}=
\begin{pmatrix}
c_w^i& & -T_w^i\\
c_w^j-T_w^j\begin{pmatrix} b^i\\0\end{pmatrix} & &-T_w^j\begin{pmatrix} A^i & B^i\\ 0_{d\times m} & \Id_{m\times m}\end{pmatrix}
\end{pmatrix}
\end{array}
\end{equation}
\subsection{Switching cells}
\label{switchcontribution}
We have to manage another constraint which comes from the cell switches. After applying the available law in cell $X^i$,
we have to specify the reachable cells i.e. the cells $X^j$ such that there exists $(x,u)$ satisfying:
\[
(x,u)\in X^i\text{ and } (A^ix+B^iu+b^i,u)\in X^j
\]
We say that a switch from $i$ to $j$ is fireable iff:
\begin{equation}
\label{Eq:feas}
\left\{(x,u)\in\rdm\left| 
\begin{array}{l} 
T_s^i(x,u)^\intercal\ll c_s^i\\ 
T_s^j (A^ix+B^iu+b^i,u)^\intercal \ll c_s^j\\ 
T_w^i(x,u)^\intercal\leq c_w^i\\ 
T_w^j (A^ix+B^iu+b^i,u)^\intercal \leq c_w^j
\end{array}\right\}\right.\neq \emptyset 
\end{equation}
We will denote by $i\to j$ if the switch from $i$ to $j$ is fireable. Recall that the symbol $<$ means that 
we can deal with both strict inequalities and inequalities. Problem~\eqref{Eq:feas} is a linear programming 
feasibility problem with both strict and weak inequalities. However, we only check whether the system 
is solvable and we can detect infeasibility by using Motzkin transposition theorem~\cite{Motzkin}. Motkin's theorem is an 
alternative type theorem, that is we oppose two linear systems such that exactly one of the two is feasible. 
To describe the alternative system, we have to separate strict and weak inequalities and use 
the matrices $E_s^{ij}$ and $E_w^{ij}$ defined at Equation~\eqref{switchmat}.
Problem~\eqref{Eq:feas} is equivalent to check whether the set $\{y=(z,x,u)\in \rr^{1+d+m}\mid E_w^{ij} y\geq 0,\ E_s^{ij} y\gg 0\}$
is empty or not. To detect feasibility we test the infeasibility of the alternative system defined as:
\begin{equation}
\label{eqMotzkin}
\left\{
\begin{array}{l}
(E_s^{ij})^\intercal p^s+(E_w^{ij})^\intercal p=0\\
\\
\sum_{k\in\mathbb{I}} p_k^s=1\\
\\
p_k^s\geq 0,\ \forall\, k\in \mathbb{I}\\
\\
p_i\geq 0,\ \forall\, i\notin \mathbb{I}
\end{array} 
\right.
\end{equation}
From Motzkin's transposition theorem~\cite{Motzkin}, we get the following proposition.
\begin{proposition}
\label{MotzkinProp}
Problem~\eqref{Eq:feas} is feasible iff Problem~\eqref{eqMotzkin} is not.
\end{proposition}
However reasoning directly on the matrices can allow unfireable switchs. For certain initial conditions,
for all $k\in\nn$, the condition $(x_k,u_k)\in X^i\text{ and } (A^i x_k+B^i u+b^i,u)\in X^j$ does not hold 
whereas Problem~\eqref{Eq:feas} is feasible. To avoid it, we must know all the possible trajectories of the 
system (which we want to compute) and remove all inactivated switchs. A sound way to underapproximate unfireable 
transitions is to identify unsatisfiable sets of linear constraints. 
\begin{example}
We continue to detail our running example. More precisely, we consider the possible switches. 
We take for example the cell $X^2$. To switch from cell $X^2$ to cell $X^1$ is possible if 
the following system of linear inequalities has a solution:
\begin{equation}
\label{eqmeet}
\begin{array}{rcc}
-9x+7y+6u&<&5\\
-0.8532x+2.5748y-10.4460&<&-68\\
-3.3662x+2.1732y-1.1084u&<&-58\\
4x-8y+8u&\leq&-4\\
u&\leq& 3\\ 
-u&\leq& 3
\end{array}
\end{equation}
The two first consists in constraining the image of $(x,y,u)$ to belong to $X^1$ 
and the four last constraints correspond to the definition of $X^2$. The representation of these 
two sets ($X^2$ and the preimage of $X^1$ by the law defined in $X^2$) is given at Figure~\ref{figuremeet}.
\begin{figure}
\includegraphics[scale=0.25]{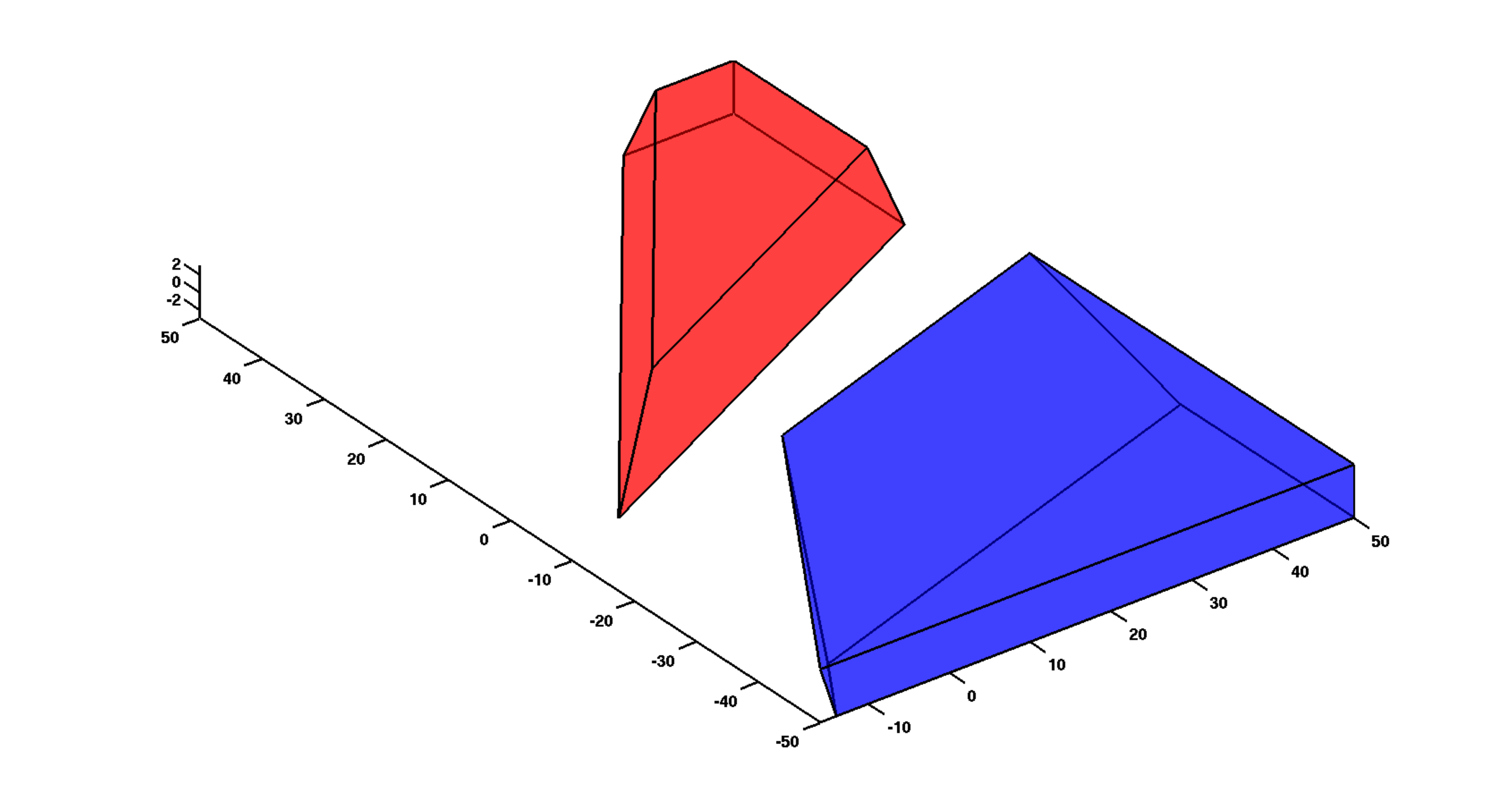}
\caption{The truncated representation of $X^2$ in red and the preimage of $X^1$ by the law inside $X^2$ in blue}
\label{figuremeet}
\end{figure}
We see at Figure~\ref{figuremeet} that the system of inequalities defined at Equation~\eqref{eqmeet} seems to not have 
solutions. We check that using Equation~\eqref{eqMotzkin} and Proposition~\ref{MotzkinProp}. The matrices $E_s^{ij}$ 
and $E_w^{ij}$ of Equation~\eqref{eqMotzkin} are in this example:
\[
E_s^{21}=\begin{pmatrix}
5 & & 9 & & -7 & & -6\\
-68 & & 0.8532 & & -2.5748 & & 10.446 \\ 
-58 & & 3.3662 & & -2.1732 & & 1.1084 
\end{pmatrix}
\text{ and }
E_w^{21}=\begin{pmatrix}
-4 & & -4 & &8 & & -8 \\ 
3 & &0 & & 0 & &-1 \\
3 & &0 & & 0 & & 1
\end{pmatrix}
\] 
We thus solve the linear program defined in Equation~\eqref{eqMotzkin} (with Matlab and Linprog) and we found 
$p=(0.8735,0.0983,0.0282)^\intercal$ and $q=(0.3325,14.2500,7.8461)^\intercal$. This means that the alternative system 
is feasible and consequently the initial is not from Proposition~\ref{MotzkinProp}. Finally the transition from 
$X^2$ to $X^1$ is not possible.
\end{example}
\subsection{Piecewise quadratic Lyapunov functions to compute invariant sets}
Now we adapt the work of Rantzer and Johansson~\cite{1205102} and the work of Mignone et al~\cite{912814} 
to compute of an invariant set for switched systems i.e. a subset $\inset$ such that $(x_k,u)\in \inset$ 
implies $(x_{k+1},u)\in \inset$. These works are instead focused on deciding whether a piecewise affine system is 
global asymptotic convergent or not. Even if the problem is undecidable~\cite{BBKPT00} the latter authors prove a 
stronger property on the system: there exists a piecewise Lyapunov functions for the piecewise affine systems.     
Rantzer and Johansson~\cite{1205102} and Mignone et al~\cite{912814} suggest to compute a piecewise 
quadratic function as Lyapunov function in the case of discrete-time piecewise affine systems to prove
GAS property. 
Recall that a piecewise quadratic function on $\rd$ is a function defined on a polyhedric partition of $\rd$ 
which is quadratic on each polyhedron of the partition. In this paper, we propose
to compute a (weaker) piecewise Lyapunov function to characterize an invariant set for our piecewise affine systems. 
In this section, we will denote by $V$ this function. The pieces are given by the cells of the piecewise affine system 
and thus $V$ is defined as:
\[
\begin{array}{ll}
V(x,u)&=V^i(x,u), \text{ if }\xu\in X^i\\
      &=\xu^\intercal P^i \xu+2{q^i}^\intercal \xu,\text{ if } \xu\in X^i
\end{array}    
\]
The function $V^i$ is thus a local function only defined on $X^i$.

A sublevel set $S_{\alpha}$ of $V$ of level $\alpha\in\rr$ is represented as:
\[
\begin{array}{lcl}
S_{\alpha}&=&\bigcup_{i\in I}S_{i,\alpha}\\
&=&\bigcup_{i\in I}\left\{\xu \in X^i\mid \xu^\intercal P^i\xu+2{q^i}^\intercal x\leq \alpha\right\}\\
&=&\bigcup_{i\in I}\left\{\xu \in X^i\mid \uxu^\intercal \begin{pmatrix} -\alpha & {q^i}^\intercal\\ 
q^i & P^i\end{pmatrix}\uxu\leq 0\right\}
\enspace .
\end{array} 
\]
The set $S_{i,\alpha}$ is thus the local sublevel set of $V^i$ associated to the level $\alpha$.

So we are looking a family of pairs of a matrix and a vector $\{(P^i,q^i)\}_{i\in I}$ and
a real $\alpha\in\rr$ such that $S_{\alpha}$ is invariant by the piecewise affine system. 
To obtain invariance property, we have to constraint $S_{\alpha}$ to contain initial conditions 
of the system. Finally, to prove that the reachable set is bounded, we have to constraint
$S_{\alpha}$ to be bounded.

Before deriving the semi-definite constraints, let us first state a useful result in Proposition~\ref{safeineq}.
This result allows to encode implications into semi-definite constraint in a safe way safe.
The implication must involve quadratic inequalities on both sides. 
\begin{proposition}
\label{safeineq}
Let $A,B,C$ be $d\times d$ matrices. Then 
$C+A+B\succeq 0$ holds implies that 
the implication
$(y^\intercal Ay\leq 0 \land\ y^\intercal By\leq 0) \implies
y^\intercal Cy\geq 0$
holds. 
\end{proposition}
\begin{proof}
Suppose that $C+A+B\succeq 0$. It is equivalent to say 
$y^\intercal (C+A+B) y\geq 0$ for all $y\in\rd$. Now 
pick $z\in\rd$ such that $z^\intercal A z\leq 0$ and 
$z^\intercal B z\leq 0$. Since $z^\intercal Cz\geq -z^\intercal A z-z^\intercal B z$, we conclude 
that $z^\intercal Cz\geq 0$ and the implication is true.
\end{proof}
\subsubsection{Writing invariance as semi-definite constraints}.
We assume that $(x,u)\in X^i\cap S_{i,\alpha}$ (this index $i$ is unique). Invariance means that if we apply the available law to $(x,u)$ and 
suppose that the image of $(x,u)$ belongs to some cell $X^j$ (notation $i\to j$), then the image of $(x,u)$ belongs to 
$S_{j,\alpha}$. Note that $(x,u)\in X^i$ 
and its image is supposed to be in $X^j$ then $(x,u)\in X^{ij}$. Let $(i,j)\in I^2$ such that $i\to j$, invariance 
translated in inequatilities and implication gives :
\begin{equation}
\label{equationstability}
\xu\in X^{ij} \land\ \xu\in S_{i,\alpha}
\implies \begin{pmatrix}A^i x+B^iu+b^i\\ u\end{pmatrix}\in S_{j,\alpha}
\end{equation}
We can use the relaxation of Subsection~\ref{quadratizationsub} as representation of cells and use 
matrix variables $W^i$ and $U^{ij}$ to encode their quadratization. We get, for $(i,j)\in I^2$ such that $i\to j$:
\begin{equation}
\label{equationstability2}
\begin{array}{ll}
&\uxu^\intercal {E^{ij}}^\intercal U^{ij} E^{ij} \uxu \geq 0 \land\
\uxu^\intercal \begin{pmatrix} -\alpha & {q^i}^\intercal\\ q^i & P^i \end{pmatrix}\uxu\leq 0\\
\implies&\uxu^\intercal \left({F^i}^\intercal \begin{pmatrix} -\alpha & {q^j}^\intercal\\q^j & P^j\end{pmatrix} 
F^i\right)\uxu \leq 0
\end{array}
\end{equation}
where $E^{ij}$ is the matrix defined at Equation~\eqref{switchquad} and  
$F^i$ is defined at Equation~\eqref{homogeneouslaw}. 

Finally, we obtain a stronger condition by considering semi-definite constraint such as 
Equation~\eqref{equationstabilityfinal}. Proposition~\ref{safeineq} proves that if $(P^i,P^j,q^i,q^j,U^{ij})$
is a solution of Equation~\eqref{equationstabilityfinal} then $(P^i,P^j,q^i,q^j,U^{ij})$ satisfies 
Equation~\eqref{equationstability2}. For $(i,j)\in I^2$ such that $i\to j$:  
\begin{equation}
\label{equationstabilityfinal}
-{F^i}^\intercal \begin{pmatrix} 0 & {q^j}^\intercal\\q^j & P^j\end{pmatrix} 
F^i+\begin{pmatrix} 0 & {q^i}^\intercal\\ q^i & P^i \end{pmatrix}
-{E^{ij}}^\intercal U^{ij} E^{ij} \succeq 0\enspace .
\end{equation}
Note that the symbol $-\alpha$ is cancelled during the computation. 
\subsubsection{Integrating initial conditions}.
To complete invariance property, invariant set must contain initial conditions. Suppose 
that initial condition is a polyhedron $X^0=\{(x,u)\in\rdm\mid T_w^0(x,u)\leq c_w^0,\ T_s^0(x,u)\ll c_s^0\}$.
We must have $X^0\subseteq S_{\alpha}$. But $X^0$ is contained in the union of $X^i$. Hence $X^0$ is the union 
over $i\in I$ of the sets $X^0\cap X^i$. If, for all $i\in I$, the set $X^0\cap X^i$ is contained in $S_{i,\alpha}$ then 
$X^0\subseteq S_{\alpha}$. We can use the same method as before to express that all sets $S_{i,\alpha}$ such that 
$X^0\cap X^i\neq \emptyset$ must contain $X^0\cap X^i$. In term of implications, it can be rewritten 
as for all $i\in I$ such that $X^0\cap X^i\neq \emptyset$:
\begin{equation}
\label{initeq}
(x,u)\in X^0\cap X^i\implies (x,u) P^i (x,u)^\intercal+2(x,u) q^i \leq \alpha
\end{equation}
Since $X^0\cap X^i$ is a polyhedra, it admits some quadratization that is:
$\overline{X^0\cap X^i}=\{(x,u)\in \rdm \mid (1,x,u){E^{0i}}^\intercal Z^i E^{0i}(1,x,u)^\intercal\geq 0\}$
where $E^{0i}=\begin{pmatrix}E_s^{0i}\\E_w^{0i}\end{pmatrix}$ with: 
\begin{displaymath}
E_w^{0i}=
\begin{pmatrix}
c_w^0 & -T_w^0\\
c_w^i & -T_w^i
\end{pmatrix} 
\text{ and }
E_s^{0i}=
\begin{pmatrix}
1 & 0_{1\times (d+m)}\\
c_s^0 & -T_s^0\\
c_s^i & -T_s^i
\end{pmatrix}
\end{displaymath} and $Z^i$ 
is some symmetric matrix whose coefficients are nonnegative.

For all $i\in I$ such that $X^0\cap X^i\neq \emptyset$, we obtain a stronger notion by introducing 
semi-definite constraints:
\begin{equation}
\label{initsdp}
-\begin{pmatrix} -\alpha & {q^i}^\intercal \\q^i & P^i\end{pmatrix}
-{E^{0i}}^\intercal Z^i E^{0i}\succeq 0
\end{equation}
Proposition~\ref{safeineq} proves that if $(P^i,q^i,Z^i)$ is a solution of Equation~\eqref{initsdp}
then $(P^i,q^i,Z^i)$ satisfies Equation~\eqref{initeq}.

Note since $X^0\cap X^i$ is a polyhedron then its emptyness can be decided by checking the feasibility 
of the linear problem~\eqref{initcell} and by using of same argument than Proposition~\ref{MotzkinProp}. 
\begin{equation}
\label{initcell}
\left\{
\begin{array}{l}
(E_s^{0i})^\intercal p^s+(E_w^{0i})^\intercal p=0\\
\\
\sum_{k\in\mathbb{I}} p_k^s=1\\
\\
p_k^s\geq 0,\ \forall\, k\in \mathbb{I}\\
\\
p_i\geq 0,\ \forall\, i\notin \mathbb{I}
\end{array} 
\right.
\end{equation}
Linear program~\eqref{initcell} is feasible iff $X^0\cap X^i=\emptyset$.
\subsubsection{Writing boundedness as semi-definite constraints}.
The sublevel $S_{\alpha}$ is bounded if and only if for all $i\in I$, the sublevel $S_{i,\alpha}$ is bounded 
The boundedness constraint in term of implications is, for all $i\in I$, there exists $\beta\geq 0$:
\begin{equation}  
\label{equationboundiness}
(x,u)\in X^i \land \xu\in S_{i,\alpha}\implies \norm{(x,u)}_2^2\leq \beta
\end{equation}
where $\norm{\cdot}_2$ denotes the Euclidian norm of $\rdm$.

As invariance, we use the quadratization of $X^i$ and the definition of $S_{i,\alpha}$. We use the fact
that $\norm{(x,u)}_2^2=\xu^\intercal \Id_{(d+m)\times (d+m)}\xu$ and we get for all $i\in I$:
\begin{equation}  
\label{equationboundiness2}
\begin{array}{ll}
&\uxu^\intercal {E^i}^\intercal W^i E^i \uxu \geq 0 \text{ and }
\uxu^\intercal \begin{pmatrix} -\alpha & {q^i}^\intercal\\q^i & P^i \end{pmatrix}\uxu\leq 0\\  
\implies &\uxu^\intercal \begin{pmatrix} -\beta & 0_{1\times (d+m)} \\
0_{(d+m)\times 1} & \Id_{(d+m)\times (d+m)}\end{pmatrix}\uxu\leq 0
\end{array}
\end{equation}
where $E^i$ is defined in Equation~\eqref{quadratization}. 

Finally, as invariance we obtain a stronger condition by considering semi-definite constraint such as 
Equation~\eqref{equationboundfinal}. Proposition~\ref{safeineq} proves that $(P^i,q^i,W^i)$ is a solution 
of Equation~\eqref{equationboundfinal} the $(P^i,q^i,W^i)$ satisfies Equation~\eqref{equationboundiness2}. 
For all $i\in I$: 
\begin{equation}
\label{equationboundfinal}
-{E^i}^\intercal W^i E^i
+\begin{pmatrix} -\alpha & {q^i}^\intercal\\q^i  & P^i \end{pmatrix}
+\begin{pmatrix} \beta & 0_{1\times (d+m)}\\
0_{(d+m)\times 1} & -\Id_{(d+m)\times (d+m)}\end{pmatrix} \succeq 0
\end{equation}
\subsubsection{Method to compute invariant set for piecewise affine systems 
and prove the boundedness of its reachable set}.
To compute a piecewise ellipsoidal invariant set for a piecewise affine systems 
of the form~\eqref{pwa} whose initial conditions is a polyhedron, we can proceed as follows:
\begin{enumerate}
\item Define a matrix $L$ of size $I\times I$ following Equation~\eqref{Eq:feas}:
set $L(i,j)=1$ if Problem~\eqref{eqMotzkin} is not feasible and $L(i,j)=0$ otherwise;
\item Define the real variables $\alpha,\beta$;
\item For $i\in I$, compute the matrix $E^i$ of Equation~\eqref{quadratization}
define the variable $P^i$ as a symmetric matrix of size $(d+m)\times (d+m)$, 
the variable matrix $W^i$ with nonnegative coefficients of size 
$(\sharp\text{ lines of } E^i)\times (\sharp\text{ lines of }E^i)$ and add the constraint
\eqref{equationboundfinal}. If Problem~\eqref{initcell} is not feasible, 
add Constraint~\eqref{initsdp};
\item For all $(i,j)\in I^2$, if $L(i,j)=1$ construct the matrix $E^{ij}$ defined by Equation~\eqref{switchquad} 
and define the symmetric matrix variable $U^{i,j}$ of the size $(\sharp\text{ lines of } E^{ij})\times 
(\sharp\text{ lines of } E^{ij})$ with nonnegative coefficients and add the constraint~\eqref{equationstabilityfinal};
\item Add as linear objective function the sum of $\alpha$ and $\beta$ to minimize; 
\item Solve the semi-definite program; 
\item If there exists a solution then the set $\bigcup_{i\in I}\{(x,u)\in X^i\mid (x,u)P_{opt}^i(x,u)^\intercal+2(x,u) 
q_{opt}^i\leq \alpha_{opt}\}$ is a bounded invariant of system~\eqref{pwa} and the norm $\norm{(x,u)}$ is less 
than $\beta_{opt}$ for all the reachable $(x,u)$ of the system.   
\end{enumerate}  
\subsection{Solution}
The method is implemented in Matlab and the solution is given by a semi-definite 
programming solver in Matlab. For our running example, Matlab returns the following 
the values:
\[
\begin{array}{cc}
\alpha_{opt}=242.0155 \\
\beta_{opt}=2173.8501 
\end{array}
\]
This means that $\norm{(x,y,u)}_2^2=x^2+y^2+u^2\leq \beta_{opt}$. We can conclude, for example, 
that the values taken by the variables $x$ are between $[-46.6154,46.6154]$. The value $\alpha_{opt}$
gives the level of the invariant sublevel of our piecewise quadratic Lyapunov function where the local 
quadratic functions are characterized by the following matrices and vectors:
\[ 
P^1=\begin{pmatrix}
 1.0181 & -0.0040 & -1.1332	\\
 -0.0040	&  1.0268 & -0.5340\\	
 -1.1332	& -0.5340 &	-13.7623
\end{pmatrix}\text{ and }
q^1=(0.1252,1.3836,-29.6791)^\intercal
\]
\[
P^2=\begin{pmatrix}
  9.1540 & -7.0159 & -2.6659 \\
 -7.0159 &  9.5054 & -2.4016 \\
 -2.6659 & -2.4016 & -8.9741 
\end{pmatrix}\text{ and }
q^2=(-21.3830,-44.6291,114.2984)^\intercal 
\]
\[
P^3= \begin{pmatrix}
  1.1555 & -0.3599 & -2.6224 \\
 -0.3599 &  2.4558 & -2.8236 \\
 -2.6224 & -2.8236 & -2.3852 \\
\end{pmatrix} \text{ and }
q^3= (-5.3138,6.7894,-40.5537)^\intercal
\]
\[
P^4=\begin{pmatrix}
  3.7314 & -3.4179 & -3.1427 \\
 -3.4179 &  6.1955 &  0.9499 \\
 -3.1427 &  0.9499 &-10.6767 \\
\end{pmatrix} \text{ and }
q^4=(28.5011,-73.5421,48.2153)^\intercal 
\]
Finally, for conciseness reason, we only give the matrix certificates for the cell $X^1$. 
First we give the matrix $W^1$ which encodes the quadratization of the guard $X^1$. 
Recall that this matrix ensures that $(x,u)\mapsto \alpha-(x,u) P^1 (x,u)^\intercal-2(x,u)q^i$ 
is nonnegative on $X^1$.  
\[   
\begin{array}{c}
W^1=\begin{pmatrix}
63.0218 & 0.0163  &  0.0217  & 12.1557  &  8.8835\\
0.0163  &  0.0000  &  0.0000  &  0.0267   & 0.0031\\
0.0217  &  0.0000   & 0.0000  &  0.0094   & 0.0061\\
12.1557  &  0.0267  &  0.0094 &   4.2011  & 59.5733\\
8.8835   & 0.0031   & 0.0061  & 59.5733  &  3.0416
\end{pmatrix}
\end{array}
\]
Secondly, we give the matrices $U^{1j}$ which encodes the quadratization of polyhedron $X^{1j}$. 
Recall that those matrices ensure that the image of $(1,x,u)$ by $F^1$ belongs to the set 
$S_{j,\alpha}$ for all $(1,x,u)$ such that $F^1(1,x,u)\in X^j$.
\[
U^{11}=\begin{pmatrix}
 0.0004	 & 0.0000 &  0.0000	&  0.0000 &  0.0000	&  0.0000 &  0.0001\\	
  0.0000 & -0.0000 & -0.0000 & -0.0000 & -0.0000 & -0.0000 & -0.0000\\	
  0.0000 & -0.0000 & -0.0000 & -0.0000 & -0.0000 &  0.0000 & -0.0000\\	
  0.0000 & -0.0000 & -0.0000 & -0.0000 & -0.0000 & -0.0000 & -0.0000\\	
  0.0000 & -0.0000 & -0.0000 & -0.0000 & -0.0000 &  0.0000 & -0.0000\\
  0.0000 & -0.0000 &  0.0000 & -0.0000 &  0.0000 &  0.0000 &  0.0000\\	
  0.0001 & 	 -0.0000 & -0.0000 & -0.0000 & -0.0000 &  0.0000 &  0.0001
 \end{pmatrix}\\
\]
\[
U^{12}=\begin{pmatrix}
 2.1068	&  0.4134	&  0.0545	&  1.4664	&  0.1882	&  2.3955	 & 2.4132\\	
  0.4134&	  0.0008	&  0.0047	&  0.0009	&  0.0819	&  0.5474&	  0.0484\\	
  0.0545&	  0.0047	&  0.0050	&  0.0147	&  0.0097	&  0.1442&	  0.2316\\	
  1.4664&	  0.0009	&  0.0147	&  0.0041	&  0.3383	&  0.8776&	  0.0999\\	
  0.1882&	  0.0819	&  0.0097	&  0.3383	&  0.0675	&  0.4405&	  0.4172\\	
  2.3955&	  0.5474	&  0.1442	&  0.8776	&  0.4405	&  8.1215&	  9.6346\\	
  2.4132&	 0.0484	&  0.2316	&  0.0999	&  0.4172	&  9.6346	&  0.9532
\end{pmatrix}\\
\]
\[
U^{13}=\begin{pmatrix}
  0.3570 &  0.2243 & 0.0031 &  0.0050	&  0.1431	&  0.0388	&  0.7675\\
  0.2243 &  0.0201	&  0.0023 &  0.0050	 & 0.1730	&  0.0494	&  0.1577\\	
  0.0031 &  0.0023	&  0.0001 &  0.0001	 & 0.0071	&  0.0006	& 0.0088\\	
  0.0050 &  0.0050	&  0.0001 &  0.0002	 & 0.3563	&  0.0009	&  0.0168\\	
  0.1431 &  0.1730	&  0.0071 &  0.3563	 & 0.0527	&  0.2689	&  0.8979\\	
  0.0388 &  0.0494	&  0.0006 &  0.0009	 & 0.2689	&  0.0137	&  0.1542\\	
  0.7675 &  0.1577	&  0.0088 &  0.0168	 & 0.8979	&  0.1542	&  0.2747
  \end{pmatrix} \\
\]
\[
U^{14}=\begin{pmatrix}
 1.3530	&  0.1912	&  0.0280	&  0.1178	&  2.9171	&  0.7079	&  1.4104\\	
  0.1912 &	  0.0512 &  0.0068	&  0.0326	&  1.7179	&  0.3764	&  0.6045\\	
  0.0280 &	  0.0068 & 0.0022	&  0.0048	&  0.1396	&  0.0264	&  0.0679\\	
  0.1178 &	  0.0326 & 0.0048	&  0.0409	&  0.5231	&  0.1204	&  0.2390\\	
  2.9171 &	  1.7179 &  0.1396	&  0.5231	& 15.0992	&  5.1148	& 14.3581\\	
  0.7079 &	  0.3764 &  0.0264	&  0.1204	&  5.1148	&  0.5102	&  1.6230\\	
  1.4104 & 	  0.6045 &  0.0679	&  0.2390	& 14.3581	&  1.6230	&  1.2985
\end{pmatrix}
\]
We remark that $U^{11}$ has negative coefficients whereas in our method,
we are looking for a nonnegative coefficients matrix. It is due to the interior point method
which is used to solve the semi-definite programming problems. 
Interior point methods returns $\epsilon$-optimal solution i.e. a solution which belongs to the ball of radius $\epsilon$
centered at an optimal solution. Hence, the solution furnished by the solver can slightly violate the constraints 
of the semi-definite program. We are aware of that and the projection of the returned solution on the 
feasible set should be studied as a future work.   


\section{Experimentations}
\label{sec:exp}

To illustrate the applicability of our method to a wide set of examples, we
generated about a thousand of dynamical systems with at most 4 partition cells,
4 state variables and a single input.

In~\cite{BBKPT00}, the authors show (Theorem 2) that to determine the stability 
a piecewise affine dynamical system is undecidable. 
In order to generated more stable examples, we restricted the class of program
generated. Each partition cell affine semantics would
be \begin{inparaenum}[(i)] \item generated with
small coefficients, since big coefficients are usually avoided in controllers
and, \item enforced locally stable when needed by updating the values of the
coefficients using the spectal radius\end{inparaenum}.

Our example synthesis still does not guaranty to obtain globally stable system,
but, with these required properties of local stability and small coefficients,
it is more likely that switching from one cell to the other would not break
stability and therefore boundedness of the reachable states. The intuition
behind is that when we pass from a cell to another cell, we multiply a vector by
a small number then all the coordinates of the vector image are strictly smaller
than the ones of initial vector.

About 300 of such 1000 examples are automatically shown to be bounded using our
technique while the class of program considered is unlikely to be analyzable
with other static analysis tools the author are aware of, including the previous
analyzes proposed~\cite{DBLP:conf/atva/RouxG13}. A typical run of the analysis,
including the time to generate the problem instance, is about 20s.

All the computation have been performed within Matlab, including the synthesis
of the examples. The source code of the analysis as well a document summarizing
the examples and their analysis is available at
\url{https://cavale.enseeiht.fr/piecewisequadratic14/}.


  


\section{Conclusion}
\label{sec:conclu}

The presented approach is able, considering a piecewise affine system, to
compute a piecewise quadratic invariant able to bound the set of reachable
state.

The technique extends the classical quadratic Lyapunov function synthesis using
SDP solvers by formulating a more complex set of contraints to the SDP
solver. This new formulation accounts the definition of the partitionning and
encodes within the SDP constraints the relationship between partitions. 

In practice our technique has been applied to a wide set of generated examples and
was able to bound their reachable state space while a global quadratic invariant
was proven not computable.

Our future work will consider the combination of this technique with other
formal methods. A first direction will rely on the computed piecewise quadratic form
as a template domain, bounding its value on some code using either Kleene
iterations~\cite{CousotCousot77-1} or policy
iteration~\cite{DBLP:journals/jsc/GawlitzaSAGG12}. This will require to extend
the existing algorithms to fit this piecewise description of the template.

A second direction is to ease the applicability of the method and to intregrate
the technique in a more common analysis framework. A requirement for the
presented work is to obtain a global representation of the program, as matrix
updates and conditions. Existing static analysis~\cite{DBLP:conf/atva/RouxG13}
used for policy iteration extracts such a graph with the appropriate
representation. We plan to integrate the two frameworks to ease the
applicability on more realistic programs in an automated fashion.  


\bibliographystyle{alpha}
\bibliography{piecewisequadraticlyapunovbib}
\end{document}